\documentclass[a4paper]{article}
\usepackage{amssymb,amsmath, amsthm,latexsym,tikz,setspace,cite,enumerate,bbm}
\usepackage[english]{babel}

\newtheorem{theorem}{Theorem}

\newtheorem{lemma}[theorem]{Lemma}

\newtheorem{conjecture}[theorem]{Conjecture}

\newcommand{\bE}{\mathbb{E}}
\newcommand{\bP}{\mathbb{P}}
\newcommand{\sci}{\chi_s'} 

\newcommand{\bigO}{\ensuremath{O}}

\newcommand{\sdeg}{d^\textsf{s}}
\newcommand{\sne}{N^\textsf{s}}
\newcommand{\sm}{\setminus}

\newcommand{\med}{\text{\rm med}}
\newcommand{\xom}{\Omega^*} 
\newcommand{\C}{C}

\title{A stronger bound for the strong chromatic index}
\author{Henning Bruhn and Felix Joos}
\date{}

\begin{document}

\maketitle

\vspace{-1cm}

\begin{abstract}
We prove $\sci(G)\leq 1.93 \Delta(G)^2$ for graphs of sufficiently large maximum degree where $\sci(G)$
is the strong chromatic index of $G$.
This improves an old bound of Molloy and Reed.
As a by-product, 
we present a Talagrand-type inequality where it is allowed to exclude unlikely bad outcomes 
that would otherwise render the inequality unusable.
\end{abstract}


\section{Introduction}

Edge colorings are well understood, for strong edge colorings this is much less the case. 
An edge coloring can be viewed as a partition of the edge set of a graph $G$
into matchings; the smallest such number of partition classes is 
the \emph{chromatic index} of $G$. If we consider the natural stronger notion of a partition
into \emph{induced} (or \emph{strong}) matchings, we arrive at the \emph{strong 
chromatic index $\sci(G)$} of $G$, 
the minimal number of induced matchings needed.

The classic result of Vizing, and independently Gupta, constrains the chromatic 
index of a (simple) graph $G$ to a narrow range: it is either equal to the trivial lower bound of the 
maximum degree $\Delta(G)$, or one more than that.
The strong chromatic index, in contrast, can vary much more. 
The trivial lower bound and 
a straightforward greedy argument give a range of 
$\Delta(G)\leq\sci(G)\leq 2\Delta(G)^2-2\Delta(G)+1$ for all graphs~$G$.
Erd\H os and Ne\v{s}et\v{r}il~\cite{FSGT89} conjectured 
a much stricter upper bound:

\newtheorem*{strongconj}{Strong edge coloring conjecture}

\begin{strongconj}
$\sci(G)\leq \frac{5}{4}\Delta(G)^2$ for all graphs $G$.
\end{strongconj}
If true, the conjecture would be optimal,  
because any blow-up of the $5$-cycle as in Figure~\ref{fig: blow up C5}
 attains equality.
For odd maximum degree, 
Erd\H os and Ne\v{s}et\v{r}il conjectured 
that 
$\sci(G)\leq \tfrac{5}{4}\Delta(G)^2-\tfrac{1}{2}\Delta(G)+1$, which again would be tight.

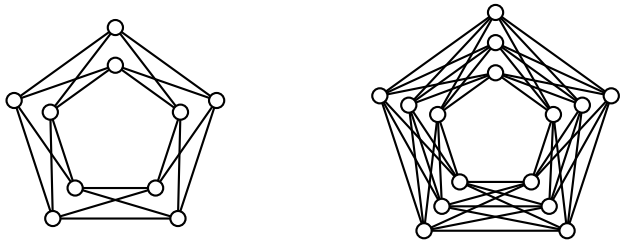
\begin{figure}[t]
\centering
\begin{tikzpicture}[every edge quotes/.style={}]
\tikzstyle{hvertex}=[thick,circle,inner sep=0.cm, minimum size=2mm, fill=white, draw=black]
\tikzstyle{hedge}=[thick]

\def\blowup{
\def\angle{360/5}
\foreach \i in {0,1,2,3,4}{
  \foreach \j in {1,...,\sizeofk}{
     \foreach \l in {1,...,\sizeofk}{
        \draw[hedge] (\angle*\i+90:\radius+\j*\stepfactor) -- (\angle*\i+\angle+90:\radius+\l*\stepfactor);
     }
  }
}

\foreach \i in {0,1,2,3,4}{
  \foreach \j in {1,...,\sizeofk}{
    \node[hvertex] (a\i\j) at (\angle*\i+90:\radius+\j*\stepfactor) {};
  }
}
}

\def\radius{0.4cm}
\def\stepfactor{0.5cm}
\def\sizeofk{2}
\blowup

\begin{scope}[shift={(5,0)}]
\def\sizeofk{3}
\def\radius{0.4cm}
\def\stepfactor{0.4cm}
\blowup
\end{scope}

\end{tikzpicture}
\caption{Two blow-ups of the $5$-cycle}\label{fig: blow up C5}
\end{figure}

In a breakthrough article of 1997,  Molloy and Reed~\cite{MR97} 
demonstrated how probabilistic coloring methods could be used to 
beat the trivial greedy bound:
\begin{align*}
	\sci(G)\leq 1.998\Delta(G)^2
\end{align*}
for graphs $G$ with $\Delta(G)$ sufficiently large.

We improve this bound:
\begin{theorem}\label{thm: str chr}
If $G$ is a graph of sufficiently large maximum degree $\Delta$,
then
\begin{align*}
	\sci(G)\leq 1.93 \Delta^2.
\end{align*}
\end{theorem}

A strong edge coloring of a graph $G$ 
may be viewed as an ordinary vertex coloring in the 
square $L^2(G)$ of the linegraph of $G$. (The square of any graph
is obtained by adding edges between any vertices of distance of most~$2$.)
Working in $L^2(G)$ permitted Molloy and Reed to split
the strong edge coloring problem
into two weaker sub-problems. First, they showed
that the 
neighborhood of any vertex in $L^2(G)$ is somewhat sparse. 
Second, 
based on a probabilistic coloring method, they 
proved a  coloring result for graphs with sparse neighborhoods,
that holds for general graphs, not only for  squares of  linegraphs.  

We follow these same steps
 but make a marked  improvement in each sub-problem: our sparsity result 
is asymptotically best-possible; and our coloring lemma needs fewer colors. 
We discuss the differences between our approach and that of Molloy and Reed 
in detail in Sections~\ref{sec: outline} and~\ref{sec: coloring}.

As a tool for our coloring lemma we develop in Section~\ref{sec: talagrand}
a version of Talagrand's inequality
that excludes exceptional outcomes. 
Talagrand's inequality is used to verify that random variables on product 
spaces are tightly concentrated around their expected value. 
It is  particularly suited for random variables that only change little 
when a single coordinate is modified. This will not be the case in our 
application: in some very rare events a single change might result 
in a very large effect. To cope with this, we formulate a version 
of Talagrand's inequality in which such large effects of tiny probability
can be ignored. We take some effort to make the application of the inequality
as simple as possible
as we have some hopes that it might be useful elsewhere.

\medskip
A weakening of the strong edge coloring conjecture 
yields a statement on \emph{strong} cliques, the cliques of 
the square of the linegraph. 
\begin{conjecture}
Any strong clique of a graph $G$ has size at most
$\omega(L^2(G))\leq \frac{5}{4}\Delta^2(G)$.
\end{conjecture}
Note that the edge set of any blow-up of the $5$-cycle is a strong clique, so that the conjecture 
would be tight. Not much is known about this seemingly easier conjecture. 
Chung, Gy{\'a}rf{\'a}s, Tuza and Trotter~\cite{CGTT90} showed that any graph $G$
that is $2K_2$-free  
has at most $\tfrac{5}{4}\Delta^2(G)$ edges. In such a graph, the whole 
edge set forms a strong clique.
Faudree, Schelp, Gy{\'a}rf{\'a}s and Tuza~\cite{FSGT90} found an upper 
bound of $(2-\epsilon)\Delta^2(G)$ for the size of any strong clique, for some
small $\epsilon$. Bipartite graphs are easier to handle in this respect: 
the same authors proved that the strong clique can never have size larger than $\Delta(G)^2$.
Again, this is tight, as  balanced complete bipartite graphs attain that bound.

We prove:
\begin{theorem}\label{thm: clique}
If $G$ is a graph with maximum degree $\Delta\geq 400$,
then its strong clique has size at most
$\omega(L^2(G))\leq 1.74 \Delta^2$.
\end{theorem}

Coming back to strong edge colorings, let us note that there are a number 
of results for special graph classes. The conjecture was verified for maximum degree~$3$
by Andersen~\cite{And92}, and independently by Horek, Qing and Trotter~\cite{HQT93}.
For $\Delta(G)=4$, Cranston~\cite{Cra06} achieves a bound of $\sci(G)\leq 22$, which 
is off by~$2$.
Mahdian~\cite{Mah00} proved that $\sci(G)\leq \frac{2\Delta(G)^2}{\log \Delta(G)}$ 
if $G$ does not contain  $4$-cycles by quite involved 
probabilistic methods.
Finally, a number of works concern degenerated graphs, the earliest of which is 
by Faudree, Schelp, Gy{\'a}rf{\'a}s and Tuza~\cite{FSGT90}
who established the bound $\sci(G)\leq 4\Delta(G)+4$ for planar graphs $G$.
Kaiser and Kang~\cite{KK14} consider a generalization of strong edge colorings, 
where alike colored edges have to be even farther apart.

The use of probabilistic methods to color graphs is explored 
 in depth in the book of Molloy and Reed~\cite{MR02}, where also many more 
references can be found. We only mention additionally the article of 
Alon, Krivelevich and Sudakov~\cite{AKS99} on coloring 
graphs with sparse neighborhoods.
However, their result only implies something nontrivial if 
the neighborhoods are much sparser than we can expect in squares of linegraphs.

Strong edge colorings seem much more difficult than edge colorings. This is 
because already induced matchings are much harder to handle than ordinary matchings. 
While the size of a largest matching can be quite precisely be determined,
it is even hard to obtain good bounds for induced matchings; see for instance~\cite{JRS14,Joo14}. 


\medskip
All our graphs are simple and finite.
We use standard graph-theoretic notation and concepts that can be found, for instance, 
in the book of Diestel~\cite{Die10}.

\section{Outline and proof of Theorem~\ref{thm: str chr}}\label{sec: outline}

A strong edge coloring of a graph $G$ is nothing else than an ordinary vertex
coloring in $L^2(G)$, the square of the linegraph of $G$. (The square of a graph 
is obtained by adding an edge between any two vertices of distance~$2$.) 
Molloy and Reed~\cite{MR97} use this simple observation to 
prove their bound on the strong chromatic index in two steps. 

First, they show that neighborhoods in $L^2(G)$ cannot be too dense. 
To formulate this more precisely denote by $\sne_e$
for any edge $e$ of $G$  the set of edges of distance
at most~$1$ to $e$, which is equivalent to saying that $\sne_e$ is the neighborhod of $e$ 
in  $L^2(G)$. We will often call $\sne_e$ the \emph{strong} neighborhood of $e$.
Molloy and Reed show that for every edge~$e$
\begin{equation}\label{MRsparsity}
\text{$\sne_e$  induces in $L^2(G)$ at most ${\textstyle\left(1-\frac{1}{36}\right)\binom{2\Delta^2}{2}}$
 edges,}
\end{equation}
where $\Delta$ is the maximum degree of $G$. 

In the second step, Molloy and Reed show that any graph with sparse neighborhoods, such 
as $L^2(G)$, can by colored with a probabilistic procedure. 

Following this strategy, we also prove a sparsity result and a coloring lemma. 

\begin{lemma}\label{density neighborhood}
Let $G$ be a graph of maximum degree $\Delta\geq 1$, and let $e$ be an edge of $G$.
Then the neighborhood $\sne_e$ of $e$ induces in $L^2(G)$ a graph of at most
$\frac{3}{2}\Delta^4+5\Delta^3$ edges.
\end{lemma}

\begin{lemma}\label{coloring lemma}
Let $\gamma,\delta\in (0,1)$ be so that 
\begin{align}\label{gamma condition}
	\gamma < \frac{\delta}{2(1-\gamma)}e^{- \frac{1}{1- \gamma}}- \frac{\delta^{3/2}}{6(1-\gamma)^2}e^{- \frac{7}{8(1-\gamma)}}.
\end{align}
Then, there is an integer $R$ so that for all $r\geq R$ it follows 
$\chi(G)\leq (1-\gamma) r$ for every 
  graph $G$ with maximum degree at most $r$ in which,  for every vertex $v$, the neighborhoods
$N(v)$ induce graphs of at most $(1-\delta)\binom{r}{2}$ edges.
\end{lemma}
Condition~\eqref{gamma condition} might be slightly hard to parse. Therefore, let us remark 
that in a range of $\delta\in(0,0.9]$ 
\[
\gamma=0.1827\cdot\delta-0.0778\cdot\delta^{3/2}
\]
satisfies the condition and is not too far away from the best-possible $\gamma$. 

Our main theorem is a direct consequence of the two lemmas.

\begin{proof}[Proof of Theorem~\ref{thm: str chr}]
Let $G$ be a graph with maximum degree $\Delta$ sufficiently large.
With Lemma~\ref{density neighborhood} we conclude that for every vertex $v$ of $L^2(G)$
the neighborhood induces a graph of at most $(\frac{3}{4}+o(1))\binom{2\Delta^2}{2}$ edges.
Therefore, we can apply Lemma~\ref{gamma condition} with 
$r=2\Delta^2$, $\delta=0.24$, and $\gamma=0.035$.
\end{proof}

While Molloy and Reed developed this very neat proof technique, our contribution
consists in improving its two constituent steps. In particular, 
in the sparsity lemma we improve Molloy and Reed's~$\tfrac{1}{36}$ in~\eqref{MRsparsity}
to roughly~$\tfrac{1}{4}$. This 
is almost as good as possible: in Section~\ref{sec: density example} we construct
graphs that asymptotically reach the upper density bound of Lemma~\ref{density neighborhood}.
Our coloring lemma also yields a  $\gamma$ that is somewhat smaller than the corresponding
$\gamma$ of Molloy and Reed. We discuss the differences  between their coloring lemma and
ours in more detail in Section~\ref{sec: coloring}.

Let us have a look at some concrete numbers. 
There is a small oversight in the proof of Molloy and Reed (a lost $2$)
that results in the actual bound of $\chi_s'(G)\leq 1.9987 \Delta(G)^2$
instead of the claimed $\chi_s'(G)\leq 1.998 \Delta(G)^2$.

To what improvements lead our lemmas? With our sparsity lemma and 
the coloring lemma of Molloy and Reed
it is possible to obtain $\chi_s'(G)\leq 1.99 \Delta(G)^2$.
Our coloring lemma then leads to the factor $1.93$ in our main theorem.
Viewing all statements as bounds of the form $\chi_s'(G)\leq (2- \epsilon)\Delta(G)^2$,
we improve $\epsilon$ by a factor of $53$.

Finally, even assuming that the coloring lemma could be vastly improved (doubtful), 
this method can never go beyond $1.73\Delta(G)^2$. The sparsity bound in Lemma~\ref{density neighborhood}, which, again, is asymptotically best-possible, does not exclude a clique of size $1.73\Delta(G)^2$ in the neighborhood $\sne_e$ in $L^2(G)$. 
That means, just considering edge densities will never yield a factor smaller than~$1.73$
(and probably not even close to that number).

\section{Density of the strong neighborhood}\label{sec: density}

In this section we prove the sparsity lemma, Lemma~\ref{density neighborhood}.
For any edge $e$ in a graph $G$, we write $\sdeg(e)$ for the degree in $L^2(G)$. 
Then $\sdeg(e)=|\sne_e|$. 

\begin{lemma}\label{strdeglem}
Let $G$ be a graph of maximum degree $\Delta$, and let $e$ be an edge of $G$. Then
\[
\sdeg(e)\leq (2-\alpha-\beta)\Delta^2-2\Delta,
\]
where $\alpha\Delta$ is the number of triangles in $G$ containing $e$, and
where $\beta\Delta^2$ is the number of $4$-cycles containing $e$
plus the number of triangles incident with exactly one endvertex of~$e$.
\end{lemma}
\begin{proof}
Let $e=uv$.
 Then $|N(u)\cup N(v)\setminus \{u,v\}|\leq 2(\Delta-1)-|N(u)\cap N(v)|= (2-\alpha)\Delta-2$. 
Every edge in $\sne_e$ has at least one of its endvertices
in $N(u)\cup N(v)$. Edges with both endvertices in $N(u)\cup N(v)$ lie in a common 
$4$-cycle with $e$ or form a triangle with either $u$ or $v$, 
and thus count towards $\beta\Delta^2$. In total we obtain
\[
\sdeg(e)\leq |N(u)\cup N(v)\setminus \{u,v\}|\cdot\Delta-\beta\Delta^2\leq ((2-\alpha)\Delta-2)\Delta-\beta\Delta^2,
\]
which gives the affirmed bound.
\end{proof}

We remark that the inequality given in Lemma~\ref{strdeglem}
is even an equality if $G$ is $\Delta$-regular.
We will frequently use the observation that the containment of an edge $f$ in $k$ $4$-cycles
implies $\sdeg(f)\leq 2\Delta^2-k$.

We will occasionally need to measure how many neighbors a vertex has in some
subset of the vertex set. So, for a vertex $v$ in a graph $G$ and some set $X\subseteq V(G)$,
we write $d_X(v)$ for $|N_G(v)\cap X|$. We will similarly use the notation $\sdeg_F(e)$
for the degree of an edge $e$ in $L^2(G)$ restricted to some edge set $F$.

\newcommand{\xycyc}{C^{X,Y}_4}

\begin{proof}[Proof of Lemma~\ref{density neighborhood}]
Without loss of generality, we may assume $G$ to be $\Delta$-regular. Indeed, 
if $G$ is not then it may be embedded into a $\Delta$-regular graph, in which 
the number of edges induced by $\sne_e$ will only be larger. (The embedding 
is a standard technique: iteratively make a second copy of $G$ and connect 
copies of vertices of too low  degree by an edge.)

Let $e=uv$ and set $X=N_G(u)\cup N_G(v)\setminus \{u,v\}$, while $Y=N_G(X)\sm(X\cup \{u,v\})$. 
Then, using the notation of Lemma~\ref{strdeglem}, we set $\alpha\Delta=|N_G(u)\cap N_G(v)|$
to be the number of triangles containing $e$, and let $\beta\Delta^2$ be the number 
of edges with both endvertices in $X$.
Observe  that $|X|=(2-\alpha)\Delta-2$. 
We furthermore note that $0\leq \alpha,\beta\leq 1$;  for $\beta$ this follows
from $2\beta\Delta^2\leq |X|\cdot\Delta$. 
We denote by $m_e$, the number of edges induced by $\sne_e$
in $L^2(G)$. Our objective is to upper-bound $m_e$.

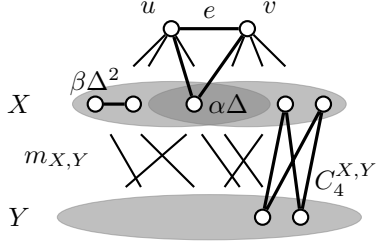
\begin{figure}[t]
\centering
\begin{tikzpicture}[every edge quotes/.style={}]
\tikzstyle{hvertex}=[thick,circle,inner sep=0.cm, minimum size=2mm, fill=white, draw=black]
\tikzstyle{hedge}=[very thick]
\tikzstyle{tedge}=[thick]

\node[hvertex,label=above left:$u$] (u) at (3,3){};
\node[hvertex,label=above right:$v$] (v) at (4,3){};

\draw[hedge] (u) -- (v) node[midway, auto]{$e$};
\draw[tedge] (u) edge (2.5,2.5) edge (2.7,2.5) edge (3.3,2.5);
\draw[tedge] (v) edge (4.2,2.5) edge (4.5,2.5) edge (3.8,2.5);

\begin{scope}[opacity=0.5]
\draw[draw=gray,fill=gray] (3,2) ellipse (1.3 and 0.3);
\draw[draw=gray,fill=gray] (4,2) ellipse (1.3 and 0.3);

\draw[draw=gray,fill=gray] (3.5,0.5) ellipse (2 and 0.3);
\end{scope}

\node at (1,2) {$X$};
\node at (1,0.5) {$Y$};
\node (mxy) at (1.5,1.25) {$m_{X,Y}$};

\node (alp) at (3.75,2){$\alpha\Delta$};

\node[hvertex] (x) at (3.3,2){};
\draw[hedge] (u) -- (x) -- (v); 

\draw[tedge] (2.2,1.6) -- (2.6,0.9);
\draw[tedge] (3.2,1.6) -- (2.4,0.9);
\draw[tedge] (2.6,1.6) -- (3.3,0.9);
\draw[tedge] (4.2,1.6) -- (3.7,0.9);
\draw[tedge] (3.7,1.6) -- (4.2,0.9);
\draw[tedge] (3.4,1.6) -- (3.9,0.9);

\node[hvertex] (x1) at (4.5,2){};
\node[hvertex] (x2) at (5,2){};
\node[hvertex] (y1) at (4.2,0.5){};
\node[hvertex] (y2) at (4.7,0.5){};
\draw[hedge] (x1) -- (y1) -- (x2) -- (y2) -- (x1);

\node (CXY) at (5.3,1){$C_4^{X,Y}$};

\node[hvertex] (x3) at (2,2){};
\node[hvertex] (x4) at (2.5,2){};
\draw[hedge] (x3) to (x4);
\node at (2,2.3){$\beta\Delta^2$};

\end{tikzpicture}
\caption{Some of the parameters used in the proof of Lemma~\ref{density neighborhood}}\label{strdeglemfig}
\end{figure}

Let us start with the estimation of $m_e$. 
Writing $\overline{\sne_e}$ for $E(G)\sm \sne_e$ we obtain
\[
2m_e=\sum_{f\in\sne_e}\sdeg_{\sne_e}(f) 
= \sum_{f\in\sne_e}\left(\sdeg(f)-\sdeg_{\overline{\sne_e}}(f)\right)
\]
To get an upper bound on $\sum_{f\in\sne_e}\sdeg(f)$ we use Lemma~\ref{strdeglem}, where
we omit the triangles and only count those $4$-cycles that 
consist entirely of edges between $X$ and $Y$. 
Let $\xycyc$ be the number of those $4$-cycles, that is, the
number of those $4$-cycles that consist only of edges with one endvertex in $X$ 
and the other in $Y$. Observe that each cycle counted by $\xycyc$ reduces the degree 
of four edges in $\sne_e$. Thus, 
$\sum_{f\in\sne_e}\sdeg(f)\leq \big(\sum_{f\in\sne_e} 2\Delta^2\big)-4\xycyc$,
by Lemma~\ref{strdeglem}. Using the lemma again, this time to upper-bound $|\sne_e|$, and substituting 
in our above estimation of $m_e$, we get
\begin{align*}
2m_e&\leq 2\Delta^2|\sne_e| -4\xycyc-\sum_{f\in\sne_e}\sdeg_{\overline{\sne_e}}(f)\\
&\leq 2\Delta^2 \left((2-\alpha-\beta)\Delta^2-2\Delta\right)
-4\xycyc-\sum_{f\in\sne_e}\sdeg_{\overline{\sne_e}}(f)\\
&= 2(2-\alpha-\beta)\Delta^4-4\Delta^3
-4\xycyc-\sum_{f\in\sne_e}\sdeg_{\overline{\sne_e}}(f)
\end{align*}

In order to obtain a lower bound on $\sum_{f\in\sne_e}\sdeg_{\overline{\sne_e}}(f)$, 
we consider paths of the form $pxyq$ in $G$, where $x\in X$, $y\in Y$ and $q\notin X$. 
The first edge $px$ then is in $\sne_e$, while the last edge $yq$ is outside. So, each 
such path $pxyq$ contributes~$1$ to $\sum_{f\in\sne_e}\sdeg_{\overline{\sne_e}}(f)$. 
Since $G$ is $\Delta$-regular, there are $\Delta-1$ such paths for each fixed $xyq$. 
Counting the number of such $xyq$ we arrive at $\sum_{y\in Y}d_X(y)(\Delta-d_X(y))$.
For later use, we give this parameter a name
\[
\gamma\Delta^3:=\sum_{y\in Y}d_X(y)(\Delta-d_X(y)),
\]
and observe that $0\leq\gamma\leq \tfrac{1}{2}$. Indeed, for any $y\in Y$, 
we have $d_X(y)(\Delta-d_X(y))\leq d_X(y)^2/4 \leq \tfrac{\Delta}{4}d_X(y)$ and consequently, 
$\gamma\Delta^3\leq \sum_{y\in Y}\tfrac{\Delta}{4}\cdot d_X(y)\leq \tfrac{\Delta}{4}\cdot 2\Delta^2$,
as there can be at most $2\Delta^2$ edges between $X$ and $Y$.

Summing up our discussion:
\[
\sum_{f\in\sne_e}\sdeg_{\overline{\sne_e}}(f)\geq (\Delta-1)\gamma\Delta^3,
\]
which leads to 
\begin{equation}\label{meest}
m_e\leq \left(2-\alpha-\beta-\frac{\gamma}{2}\right)\Delta^4-2\xycyc+\left(\frac{\gamma}{2}-2\right)\Delta^3.
\end{equation}

It remains to estimate $\xycyc$. For this, 
let us first compute $m_{X,Y}$, the number of edges with one endvertex in $X$ and the other in $Y$.
On the one hand, we have

\[
\sum_{x\in X}d_G(x)=m_{X,Y}+2\beta\Delta^2+|X|,
\] 
while $\Delta$-regularity, on the other hand, gives us
\(
\sum_{x\in X}d_G(x)=\Delta\cdot|X|.
\)
Together with $|X|=(2-\alpha)\Delta-2$, this implies
\begin{equation}\label{mXY}
m_{X,Y}=(2-\alpha-2\beta)\Delta^2-(4-\alpha)\Delta+2.
\end{equation}

For $x_1,x_2\in X$ define $c(x_1,x_2)$ to be the number of common neighbors of 
$x_1$ and $x_2$ in $Y$, where we put $c(x_1,x_2)=0$ if $x_1=x_2$.
Then
\[
\xycyc=\sum_{x_1,x_2\in X}\binom{c(x_1,x_2)}{2}.
\]
We compute
\[
\sum_{x_1,x_2\in X}c(x_1,x_2) = \sum_{y\in Y}\binom{d_{X}(y)}{2} 
= \frac{1}{2} \sum_{y\in Y} (d_X(y))^2-\frac{1}{2}m_{X,Y}.
\]
Using the definition of $\gamma$, this gives
\begin{align*}
	\sum_{x_1,x_2\in X}c(x_1,x_2) 
	=&\frac{1}{2}\left( \Delta\sum_{y\in Y} d_{X}(y) -\gamma\Delta^3 \right)-\frac{1}{2}m_{X,Y}\\
	=&\frac{1}{2}\left( (\Delta-1) m_{X,Y} -\gamma\Delta^3 \right).
\end{align*}
Using $m_{X,Y}\leq 2\Delta^2$ as well as~\eqref{mXY}, we get a lower and an upper bound:
\begin{equation}\label{commonest}
\frac{1}{2}\left( \left(2-\alpha-2\beta- \gamma\right)\Delta^3-6\Delta^2\right)
\leq
\sum_{x_1,x_2\in X}c(x_1,x_2) \leq \left(1-\frac{\gamma}{2}\right)\Delta^3
\end{equation}
We come back to the calculation of $\xycyc$:
\begin{align}
\notag	\xycyc&= \sum_{x_1,x_2\in X}\binom{c(x_1,x_2)}{2} = 
\frac{1}{2}\sum_{x_1,x_2\in X}\left( c(x_1,x_2)^2-c(x_1,x_2)\right)\\
& \geq \frac{1}{2}\sum_{x_1,x_2\in X} c(x_1,x_2)^2 - \frac{1}{2}\left(1-\frac{\gamma}{2}\right)\Delta^3, \label{C4ineq}
\end{align}
where we used~\eqref{commonest}.
We use the Cauchy-Schwarz inequality:
\begin{align*}
\sum_{x_1,x_2\in X} c(x_1,x_2)^2
&\geq \binom{|X|}{2} \left(\frac{\sum_{x_1,x_2\in X}c(x_1,x_2)}{\binom{|X|}{2}}\right)^2\\
&\geq 2\cdot\frac{\left(\sum_{x_1,x_2\in X}c(x_1,x_2)\right)^2}{|X|^2}\\
&\geq 2\cdot\frac{\left(\,\frac{1}{2}\left((2-\alpha-2\beta-\gamma)\Delta^3-6\Delta^2\right)\,\right)	^2}{(2-\alpha)^2\Delta^2},
\end{align*}
where the last inequality is because of $|X|=(2-\alpha)\Delta-2$ 
and~\eqref{commonest}.
We continue
\begin{align*}
\sum_{x_1,x_2\in X} c(x_1,x_2)^2
&\geq \frac{1}{2}\left(\frac{(2-\alpha-2\beta-\gamma)^2\Delta^4}{(2-\alpha)^2}-12\cdot\frac{(2-\alpha)\Delta^3}{(2-\alpha)^2}\right)\\
&\geq \frac{(2-\alpha-2\beta-\gamma)^2}{2(2-\alpha)^2}\Delta^4-6\Delta^3.
\end{align*}
We substitute this in our estimation~\eqref{C4ineq} of $\xycyc$:
\begin{align*}
\xycyc&\geq \frac{1}{2}\left(\frac{(2-\alpha-2\beta-\gamma)^2\Delta^4}{2(2-\alpha)^2}-6\Delta^3\right)
- \frac{1}{2}\left(1-\frac{\gamma}{2}\right)\Delta^3\\
&= \frac{1}{2}\left(
\frac{(2-\alpha-2\beta-\gamma)^2\Delta^4}{2(2-\alpha)^2}- 
 \left(7-\frac{\gamma}{2}\right)\Delta^3
\right).
\end{align*}
We can finally complete our estimation~\eqref{meest} of $m_e$ to:
\[
m_e\leq \left(2-\alpha-\beta - \frac{\gamma}{2}\right)\Delta^4   - \frac{\left(2-\alpha-2\beta- \gamma\right)^2}{2(2-\alpha)^2}\Delta^4+
5\Delta^3.
\]
To see that this gives
\begin{align*}
	m_e \leq \frac{3}{2}\Delta^4+5\Delta^3,
\end{align*}
observe first that we may assume that $0\leq \alpha\leq \frac{1}{2}$ or we are already done.
Setting $x= \beta + \frac{\gamma}{2}$,  we may also assume $x\leq \frac{1}{2}$.
Let $f(\alpha,x)= 2-\alpha -x - \frac{\left(2-\alpha-2x\right)^2}{2(2-\alpha)^2}$.
Note that $\frac{\partial f(\alpha,x)}{\partial \alpha}=-1+ \frac{2x(2-\alpha-2x)}{(2- \alpha)^3}<0$ for $0\leq x\leq \frac{1}{2}$.
Thus we may assume that $\alpha=0$. Consequently, it remains to verify 
that $2-x- \frac{(2-2x)^2}{8}\leq \frac{3}{2}$, which is an elementary task.
\end{proof}

We note that a good number of elements used in the proof appear already 
in the article of Molloy and Reed~\cite{MR97}:  triangles through~$e$, 
 edges in $X$, paths $xyq$ and  
$4$-cycles between $X$ and $Y$. 
Our contribution consists in parameterising these elements and then combining 
the parameters in 
the right way to give a nearly tight bound.

Lemma~\ref{density neighborhood} naturally gives an upper bound on the size of the largest clique in $L^2(G)$.
Although cliques have a very simple structure 
we were not able to push the bound given in Theorem~\ref{thm: clique} significantly further.

\begin{proof}[Proof of Theorem~\ref{thm: clique}]
Let $K$ be a largest strong clique, that is, a largest clique in $L^2(G)$,
and let $\kappa\Delta^2$ be its size. 
If $e\in K$ is an edge of $G$, then its neighborhood induces by Lemma~\ref{density neighborhood}
a graph of at most $\frac{3}{2}\Delta^4+5\Delta^3$ edges. Thus
\[
\binom{\kappa\Delta^2-1}{2}\leq \frac{3}{2}\Delta^4+5\Delta^3,
\]
which implies 
$\kappa\leq\frac{3}{2\Delta^2}+\sqrt{3+\frac{10}{\Delta}+\frac{9}{4\Delta^4}}<1.74$
for $\Delta\geq 400$.
\end{proof}

\section{The sparsity lemma is best-possible}\label{sec: density example}

In this section we describe a family of graphs
with  strong neighborhoods that almost reach the upper density 
bound of Lemma~\ref{density neighborhood}.
To show this, we turn to \emph{Hadamard codes}. 
For every $k\geq 2$, the corresponding Hadamard code
consists of $2\cdot 2^k$ $0,1$-strings of length $2^k$ each with certain properties.
For instance, for $k=2$, 
the Hadamard code is
\[
\{\textsf{{\small 1111,0000,1100,0011,0110,1001,1010,0101}}\}
\]
Notably, the code always contains the all-$0$-string and the all-$1$-string. 
We drop these and interpret the remaining code words  
as subsets of some fixed ground set of $n=2^k$ elements
$\{x_1,\ldots, x_n\}$.
Then we can see the Hadamard code as a set $\mathcal H_k$ of subsets  
of $\{x_1,\ldots, x_n\}$ with the following properties:
\begin{enumerate}[\rm (i)]
\item \label{Hadauni}$\mathcal H_k$ contains $2(n-1)$ member-sets, each of which has cardinality $\tfrac{n}{2}$
\item \label{Hadareg}every element $x_i$ lies in precisely $n-1$ member-sets of $\mathcal H_k$
\item \label{Hadaint}if $S,S'\in\mathcal H_k$ are two distinct member-sets of $\mathcal H_k$
then $|S\cap S'|\in\{0,\tfrac{n}{4}\}$.
\end{enumerate}
For a proof and more details, see for example van Lint~\cite{vLi99}.

Based on $\mathcal H_k=\{S_1,\ldots,S_{2n-2}\}$ we construct a graph $G$.
Make a copy $x'_i$ for each $x_i$, and 
put $X=\{x_1,\ldots, x_n,x'_1,\ldots, x'_n\}$. Let, furthermore, $Y=\{y_1,\ldots, y_{2n-2}\}$
be a set of $2n-2$ vertices, that is disjoint from $X$. We define a graph $G$ on 
$X\cup Y\cup\{u,v\}$, where $u,v$ are two vertices outside $X\cup Y$. 
First off, make $u$ complete to $x_1,\ldots, x_n$ and $v$ complete to $x'_1,\ldots, x'_n$,
and make $u$ adjacent to $v$. Now, let us add a bipartite graph on $X\cup Y$. 
For this, we make each $y_j$ adjacent to every vertex in $S_j$ 
and to the set $S'_j$ of their copies. See Figure~\ref{figfig} for an illustration.
With this, each vertex in $X$ has $(n-1)+1$ neighbors, 
by property~\eqref{Hadareg} of $\mathcal H_k$,
and each $y_j$ has degree $2\tfrac{n}{2}=n$, by property~\eqref{Hadauni}.
Each of $u$ and $v$ has degree $n+1$. Thus the maximum degree is $\Delta=n+1$.
We note, furthermore, that there are $(2n-2)n=2\Delta^2-\bigO(\Delta)$ edges between $X$ and $Y$.

\begin{figure}[t]
\centering
\begin{tikzpicture}[every edge quotes/.style={}]
\tikzstyle{hvertex}=[thick,circle,inner sep=0.cm, minimum size=2mm, fill=white, draw=black]
\tikzstyle{hedge}=[thick]

\def\twinstep{0.3cm}
\def\stepsize{1.4cm}
\def\ystepsize{1cm}
\def\hoehe{1.6cm}
\def\xdilat{2.5*\ystepsize-1.5*\stepsize}

\foreach \i in {0,1,2,3}{
  \node[hvertex] (x\i) at (\xdilat+\i*\stepsize-\twinstep,\hoehe) {};
  \node[hvertex] (xx\i) at (\xdilat+\i*\stepsize+\twinstep,\hoehe) {};
}

\foreach \j in {0,1,2,3,4,5}{
  \node[hvertex] (y\j) at (\j*\ystepsize,0){};
}

\node at (-0.5,0){$Y$};
\node at (-0.5,\hoehe){$X$};

\node[hvertex,label=above left:$u$] (u) at (2.5*\ystepsize-0.5cm,1.7*\hoehe){};
\node[hvertex,label=above right:$v$] (v) at (2.5*\ystepsize+0.5cm,1.7*\hoehe){};

\draw[hedge] (u) -- (v);

\foreach \i in {0,1,2,3}{
  \draw[hedge] (u) -- (x\i);
  \draw[hedge] (v) -- (xx\i);
}

\draw[hedge] (x0) -- (y0) -- (xx0); 
\draw[hedge] (x1) -- (y0) -- (xx1);

\draw[hedge] (x2) -- (y1) -- (xx2); 
\draw[hedge] (x3) -- (y1) -- (xx3);

\draw[hedge] (x0) -- (y2) -- (xx0); 
\draw[hedge] (x2) -- (y2) -- (xx2);

\draw[hedge] (x1) -- (y3) -- (xx1); 
\draw[hedge] (x3) -- (y3) -- (xx3);

\draw[hedge] (x1) -- (y4) -- (xx1); 
\draw[hedge] (x2) -- (y4) -- (xx2);

\draw[hedge] (x0) -- (y5) -- (xx0); 
\draw[hedge] (x3) -- (y5) -- (xx3);

\node[label=below:\textsf{\small 1100}] at (y0){};
\node[label=below:\textsf{\small 0011}] at (y1){};
\node[label=below:\textsf{\small 1010}] at (y2){};
\node[label=below:\textsf{\small 0101}] at (y3){};
\node[label=below:\textsf{\small 0110}] at (y4){};
\node[label=below:\textsf{\small 1001}] at (y5){};

\end{tikzpicture}
\caption{The graph $G$ for $k=2$}\label{figfig}
\end{figure}
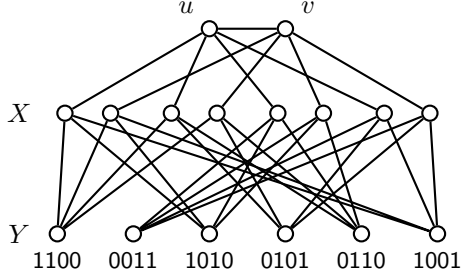

We calculate the number of $4$-cycles in $G-u-v$. 
Observe that for all $i<j$ 
the number $c(y_i,y_j)$ of common neighbors of $y_i$ and $y_j$ 
 is at most $2\cdot \tfrac{n}{4}$ (by property~\eqref{Hadaint}).
Thus, the number of $4$-cycles in $G-u-v$ is
\begin{align*}
\sum_{i<j}\binom{c(y_i,y_j)}{2} \leq \binom{2n-2}{2}\cdot \binom{n/2}{2}  
=& \frac{n^4}{4}+\bigO(n^3)= \frac{\Delta^4}{4}+\bigO(\Delta^3).
\end{align*}
Since there are only $2\Delta-1$ edges incident with $u$ or $v$,
deleting all these edges will only result in a loss of $\bigO(\Delta^3)$
edges in $L^2(G)$. In the square of the linegraph of $G-u-v$, however, 
only the $4$-cycles counted above reduce the degree $\sdeg(e)$ 
away from the maximal possible value $2\Delta^2$ (recall the argumentation in the proof of Lemma~\ref{strdeglem}).
Thus, 
recalling that there are $2\Delta^2-\bigO(\Delta)$ edges between $X$
and $Y$, we obtain that the number $m_{uv}$ of edges induced by $\sne_{uv}$ in $L^2(G)$ satisfies
\[
2m_{uv} \geq 2\Delta^2\cdot (2\Delta^2-\bigO(\Delta)) - 4\left(\frac{\Delta^4}{4}+\bigO(\Delta^3)\right) - \bigO(\Delta^3) 
\]
(Note that each $4$-cycle reduces the degree of four edges in $L^2(G)$.) Therefore
\[
m_{uv}\geq \frac{3}{2}\Delta^4-\bigO(\Delta^3),
\]
which asymptotically coincides with the bound of Lemma~\ref{density neighborhood}.

\section{The coloring procedure}\label{sec: coloring}

Molloy and Reed~\cite{MR97} described a probabilistic method, called 
\emph{naive coloring procedure} in~\cite{MR02}, to partially color a sparse graph. 
In the procedure,  every vertex first receives 
a color independently and uniformly at random. Then follows a conflict 
resolution step: whenever two adjacent vertices received the same color, 
both become uncolored. The brilliant insight here is that, with a non-negligible 
probability, in every neighborhood a good number of colors appear at least 
twice at the end of this procedure. In this way, the partial coloring 
\emph{saves} colors in comparison to a $(\Delta(G)+1)$-coloring.
This, then makes it possible to complete 
the partial coloring greedily to a full coloring. 

In the easiest manifestation of the procedure, whenever there is a conflict, 
that is, whenever two adjacent vertices receive the same color, both vertices lose their 
color. Molloy and Reed hint at a possible a improvement that is less wasteful: 
if additionally random weights on vertices are chosen,  only the vertex of lower 
weight need to be uncolored. However, the analysis becomes much 
more tedious. We propose a different conflict resolution mechanism that allows
for a simpler analysis. In addition, we not only count colors that are 
used exactly twice in the neighborhood of a fixed vertex but 
also take into account multiple occurrences. 
Molloy and Reed mentioned such a modification, too, but did not discuss it 
in detail.

\medskip
We now describe the modified Molloy and Reed coloring procedure. 
For this, assume $G$ to be an $r$-regular graph, which 
we will color with $\C=\lceil (1-\gamma)r\rceil$ colors, where $\gamma\in(0,1)$
is a constant. 
The following random experiment is performed.
\begin{enumerate}
	\item Color every vertex uniformly and independently at random from the set $\{1,\ldots,\C\}$.
	\item Choose for every edge $uv$ independently and uniformly at random an orientation $d_{uv}$.
	\item If $uv$ is an edge and $u$ and $v$ received the same color, 
	then uncolor $u$ if $d_{uv}$ points towards $u$ and uncolor $v$ otherwise.
\end{enumerate}

The key parameter is the number of colors that are saved
in the neighborhood of an arbitrary vertex $u$, that is, the number
of vertices that are assigned a color that is already used by some other
vertex in the neighborhood. To estimate this quantity, we introduce the random variable $P_u$
that counts the number of pairs of non-adjacent vertices in $N(u)$ that have, 
after the uncoloring step, the same color. 
To control overcounting, we also define $T_u$, the number of triples of distinct
non-adjacent vertices $v,w,x$ in $N(u)$ that have the same color at the end of 
the coloring procedure.

Most of the effort in the proof of Theorem~\ref{coloring lemma} will 
be spent on estimating the expected values of $P_u$ and $T_u$. 
Later, in Section~\ref{sec: proof concentration}, we will show that the random variables are  
strongly concentrated around their expectation:

\begin{lemma}\label{concentration}
Let $\gamma\in(0,1)$.
There is an $R$ so that for every $r$-regular graph $G$ with $r\geq R$ 
it holds that
\begin{align*}
\bP\!\left[\,|P_u-\bE[P_u]|\geq\sqrt{r}\log^3 r\,\right]&\leq r^{- \frac{1}{2}\log\log r}\quad\text{ and}\\
\bP\!\left[\,|T_u-\bE[T_u]|\geq\sqrt{r}\log^4 r\,\right]&\leq r^{- \frac{1}{2}\log\log r}
\end{align*}
when the coloring procedure is performed with $\C=\lceil (1-\gamma)r\rceil$ colors.
\end{lemma}

In order to prove the existence of a global partial coloring with certain nice properties,
we use the Lov\'asz Local Lemma to deduce this from our local structure.

\newtheorem*{lllem}{Lov\'asz Local Lemma}

\begin{lllem}
Let $p\in [0,1)$, and 
 $\mathcal A$ be a finite set of events so that for every $A\in\mathcal A$
\begin{enumerate}[\rm (i)]
\item $\bP[A]\leq p$; and
\item $A$ is independent of all but at most $d$ of the other events in $\mathcal A$.
\end{enumerate} 
If $4pd\leq 1$, then the probability that none of the events in $\mathcal A$ 
occur is strictly positive.
\end{lllem}

\begin{proof}[Proof of Lemma~\ref{coloring lemma}]
How large $R$ has to be will become clear at the end of the proof.
Consider some graph $G$ with maximum degree at most~$r$ satisfying the density condition 
on the neighborhoods. 
We may assume that $G$ is $r$-regular, otherwise we embed $G$, 
as in Lemma~\ref{density neighborhood},
 in an $r$-regular graph while keeping the local density condition.

We will color $G$ with $\C=\lceil(1-\gamma)r\rceil$ colors, which may be slightly more 
that the claimed bound of $\chi(G)\leq (1-\gamma) r$. However, by choosing $R$ and 
thus $r$ large enough we can find a $\gamma'$ that still satisfies~\eqref{gamma condition}
and for which $\lceil(1-\gamma')r\rceil\leq (1-\gamma) r$ for every $r\geq R$.
Thus, if necessary, we may replace $\gamma$ by $\gamma'$ in what follows. 

Pick some vertex $u$ of $G$.
We start with the estimation of $P_u$, the number of non-adjacent neighbors
of $u$ with the same (final) color. 
To this end,  consider two non-adjacent neighbors $v$ and $w$ of $u$.  
Note first that the probability that $v$ and $w$ receive the same color in step~1 is
equal to $\tfrac{1}{\C}$. Assuming that this is the case, we observe that
in order for $v$ and $w$ to keep their color in step~3, all the edges between $\{v,w\}$
and the neighbors of the same color as $v,w$ have to be chosen in step~2 so that they
point away from $\{v,w\}$. Thus, if $v$ and $w$ have $k$ common neighbors, the
probability that they both keep their (common) color is 
\[
 \left(1- \frac{1}{2\C} \right)^{2r-2k} \left(1- \frac{3}{4}\cdot\frac{1}{\C} \right)^{k},
\]
where the factor~$\tfrac{3}{4}$ stems from the fact that, out of four 
orientations of the two edges between $v,w$ and a common neighbor,
only one of these allows both $v,w$ to keep their color. 

In total, we get
\begin{align*}
	&\bP[v,w \text{ received the same color and both stay colored}]\\
	&\quad= \frac{1}{\C} \left(1- \frac{1}{2\C} \right)^{2r-2k} \left(1- \frac{3}{4}\cdot\frac{1}{\C} \right)^{k}
	=\frac{1}{\C} \left(1- \frac{1}{2\C} \right)^{2r} 
	\left(\frac{1- \frac{3}{4\C} }{1- \frac{1}{\C} + \frac{1}{4\C^2} }\right)^{k}.
\end{align*}
Note that, for any $\C\geq 1$, we have 
\[
\frac{1- \frac{3}{4\C} }{1- \frac{1}{\C} + \frac{1}{4\C^2} } \geq 1.
\]
Thus, if $r$ is large enough, then
\begin{align*}
	&\bP[v,w \text{ received the same color and both stay colored}]\\
&\qquad \geq \frac{1}{\C} \left(1- \frac{1}{2\C} \right)^{2r}  
= \frac{1}{\lceil(1-\gamma)r\rceil} \left(1- \frac{1}{2\lceil(1-\gamma)r\rceil} \right)^{2r}\\
	&\qquad \geq (1+o(1)) \frac{1}{(1-\gamma)r}e^{- \frac{1}{1-\gamma}}
\end{align*}
As there are at least $\delta\binom{r}{2}$ many pairs of non-adjacent vertices 
in $N(u)$, the calculation above implies
\begin{align*}
	\bE[P_u]&\geq (1+o(1)) \delta \binom{r}{2}\frac{1}{(1-\gamma)r}e^{- \frac{1}{1-\gamma}}\\
	&=(1+o(1))\frac{\delta r}{2(1-\gamma)}e^{- \frac{1}{1-\gamma}}.
\end{align*}

We will need a similar estimation for triples of non-adjacent neighbors of $u$. 
So, assume $v,w,x$ to be three vertices in $N(u)$ that are pairwise non-adjacent. 
For $1\leq i \leq 3$, let $k_i$ be the number of vertices with $i$ neighbors in $\{v,w,x\}$.
Since $G$ is $r$-regular,
$k_1+2k_2+3k_3=3r$.

The probability that all three of $v,w,x$ receive the same color is $\tfrac{1}{\C^2}$. 
The probability that all three retain their color is computed in a similar
way as above, where it should be noted that there are now eight possibilities
for the orientations of the three edges between $v,w,x$ and a neighbor common to all of them. 

Using the binomial series $(1+z)^\alpha=\sum_{k=0}^\infty\binom{\alpha}{k}z^k$
and $\tfrac{1}{\C^\alpha}=o(1/r)$ for $\alpha>1$, we get
\begin{align*}
	&\bP[v,w,x \text{ received the same color and stay colored}]\\
	&= \frac{1}{\C^2} \left(1- \frac{1}{2\C} \right)^{k_1}\left(1- \frac{3}{4\C} \right)^{k_2} 
	\left(1-  \frac{7}{8\C} \right)^{k_3}\\
	&= \frac{1}{\C^2} \left(1- \frac{3}{2\C} + o(1/r) \right)^{k_1/3}\left(1- \frac{9}{8\C}
+o(1/r) \right)^{2k_2/3} 
	\left(1- \frac{7}{8C} \right)^{k_3}\\
	&\leq (1+o(1))\frac{1}{(1-\gamma)^2r^2} \left(1- \frac{7}{8(1-\gamma)r} \right)^{k_1/3+2k_2/3+k_3}\\
	&= (1+o(1)) \frac{1}{(1-\gamma)^2r^2}e^{- \frac{7}{8(1-\gamma)}}.
\end{align*}

By using that every graph with $\delta \binom{r}{2}$ edges 
contains at most $\frac{\delta^{3/2}r^3}{6}$ many distinct triangles
(see, for instance, Rivin~\cite{Riv02}),
we obtain
\begin{align*}
	\bE[T_u]&\leq (1+o(1)) \frac{\delta^{3/2}r^3}{6}\frac{1}{(1-\gamma)^2r^2}e^{- \frac{7}{8(1-\gamma)}}\\
	&=(1+o(1))\frac{\delta^{3/2}r}{6(1-\gamma)^2}e^{- \frac{7}{8(1-\gamma)}}.
\end{align*}

By the inclusion--exclusion principle,
we save at least  $P_u-T_u$ many colors in $N(u)$,
that is, the number of colored vertices minus the number of colors actually 
used in $N(u)$ is at least $P_u-T_u$. 
Now, if $P_u-T_u\geq\gamma r$, 
then the number of uncolored vertices is smaller than the number of unused colors in $N(u)$.
Thus, if we can show that, with positive probability, 
it holds that $P_u-T_u\geq \gamma r$, then we can color all remaining uncolored 
vertices greedily, which then concludes the proof. 

To show this, 
let $A_u$ be the event that 
$$P_u-T_u\leq 
	\left(1- \frac{1}{\log r}\right) \left(\frac{\delta}{2(1-\gamma)}e^{- \frac{1}{1-\gamma}}
	-\frac{\delta^{3/2}}{6(1-\gamma)^2}e^{- \frac{7}{8(1-\gamma)}}\right)r.
$$
Since both random variables $P_u$ and $T_u$ are highly concentrated (see Lemma \ref{concentration}),
it follows for large $r$ that
\begin{align*}
	\bP\left[A_u\right]\leq O(r^{- \frac{1}{2}\log \log r}).
\end{align*}
Note that the event $A_u$ only depends on $A_{u'}$ if there is some vertex $z$ 
of distance at most~$2$ to both $u$ and $u'$. Thus, $A_u$ is independent of all other 
$A_{u'}$ except for a number of these that is polynomial in~$r$. 
We deduce, therefore, from 
the Lov\'asz Local Lemma that
there is a coloring of the vertices of $G$ such that no $A_u$ holds.
This, however, implies together with~\eqref{gamma condition} that $P_u-T_u\geq\gamma r$,
provided that $r$ is large enough. 
\end{proof}

\section{How to possibly  save more colors}

The factor of $1.93\Delta^2$ of Theorem~\ref{thm: str chr} is still very far from 
the conjectured factor of $1.25\Delta^2$. While it seems doubtful that probabilistic coloring
can ever get very close to $1.25\Delta^2$, there is still some hope that 
with more sophisticated arguments the factor can be improved. For us, two observations fuel this hope.

First, the two steps, the sparsity lemma and the coloring lemma, are completely dissociated. 
That is, the coloring lemma only exploits the sparsity of strong neighborhoods
but uses no structural information whatsoever. Surely, not forgetting that the task 
consists in coloring edges should help!

Second, while the sparsity lemma, Lemma~\ref{density neighborhood} is asymptotically
tight, the conjectured extreme example for the strong coloring 
conjecture, the blow-up of the $5$-cycle, 
has much sparser strong neighborhoods. 

To be more concrete, consider the blow-up in which every vertex of the $5$-cycle is 
replaced by a stable set of size~$k$, so that the maximum degree becomes $\Delta=2k$. 
Then every edge $e$ has as strong neighborhood all of the rest of the graph, and 
consequently, the strong neighborhood induces about $\tfrac{25}{32}\Delta^4$ 
edges in the square of the linegraph, which is much less than the $\tfrac{3}{2}\Delta^4$
of  Lemma~\ref{density neighborhood}. In some sense, this is not surprising because
already the degree $\sdeg(e)$ of $e$ in the square of the linegraph is much smaller than 
the $2\Delta^2$ we are working with, namely it is only about $\tfrac{5}{4}\Delta^2$. 
In conclusion, the blow-up of the $5$-cycle is quite different from 
what we assume in the sparsity lemma. 

What kind of effects could be at work  that explain this difference?
In  Lemma~\ref{density neighborhood} we assume that the edge $e$ has  
degree $\sdeg(e)$ close to $2\Delta^2$ and at the same time a very dense neighborhood. While it is 
possible to have such edges, this cannot be the case for all edges. Indeed, 
for the strong neighborhood $\sne_e$ to induce many edges, there have 
to be many $4$-cycles in $\sne_e$. (This is simply because $X$ and $Y$, the first
and second neighborhoods of the endvertices of $e$, cannot be very large so 
that the bipartite graph between $X$ and $Y$ is very dense.) 
However, every $4$-cycle reduces in $L^2(G)$ the degree of any of its edges by~$1$,
and consequently, there should be many edges $f$ in $\sne_e$ of lower degree $\sdeg(f)$
than  $2\Delta^2$. 

Obviously, an edge $f$ of low degree $\sdeg(f)$ is to our advantage, as we can 
always color $f$ at the very end when all high degree edges are already colored. 
But if we defer coloring of low degree edges then any high degree edge $e$ with many 
low degree edges in its strong neighborhood has, morally, low degree as well.
Unfortunately, we did not manage to exploit these observations in such a way
that they result in a substantial improvement.

\section{Talagrand's inequality and exceptional outcomes}\label{sec: talagrand}

To finish the proof of Lemma~\ref{coloring lemma}, we still need to show
that the probability of $P_u$ or $T_u$ deviating significantly 
from their expected values is very small. To prove 
that a random variable on a product probability space is strongly concentrated around
its expectation is a very common task, and consequently, 
 a number of powerful tools have been developed for this, among them 
McDiarmid's, Azuma's or Talagrand's inequality~\cite{McD89,Azu67,Tal95}. 
All of these tools have in common
that they require the random variable to be somewhat smooth.

Consider a family of probability spaces $((\Omega_i,\Sigma_i,\bP_i))_{i=1}^n$,
and let $(\Omega,\Sigma,\bP)$ be their product. 
One common smoothness assumption 
for a random variable $X:\Omega\to\mathbb R$ is that
\emph{each coordinate has effect at most~$c$}: 
whenever any two $\omega,\omega'\in\Omega$ differ in exactly one coordinate
then  $|X(\omega)-X(\omega')|\leq c$. 

For McDiarmid's or Azuma's inequality to give strong concentration in this 
situation, the effect~$c$ has to be small compared to $n$. If that is not the 
case, then Talagrand's inequality might still be useful. 
We describe a weaker version that is easier to apply than the full inequality. 

We say that $X$ 
\emph{has certificates of size $s$ for exceeding value $k$} if for any 
$\omega\in\Omega$ with $X(\omega)\geq k$, there is a set $I$ of at most $s$
coordinates such that also $X(\omega')\geq k$ for any $\omega'\in\Omega$ 
with $\omega|_I=\omega'|_I$.
The following version of Talagrand's inequality appears in Molloy and Reed~\cite[p.~234]{MR02}:

\begin{theorem}[Talagrand]
Let $((\Omega_i,\Sigma_i,\bP_i))_{i=1}^n$ be probability spaces, and let $(\Omega,\Sigma,\bP)$ be their product space.
Let $X:\Omega\to\mathbb R$ 
be a non-negative random variable with $X\neq 0$ 
so that each coordinate has effect at most $c$, and assume 
$X$ to have, for any $k$, certificates of size at most $k\ell$ for exceeding $k$. 
Then for any $0\leq t\leq\bE[X]$: \[
\bP\left[|X-\bE[X]|>t+60c\sqrt{\ell\bE[X]}\right] \leq 4e^{-\frac{t^2}{8c^2\ell\bE[X]}}.
\] 
\end{theorem}

Unfortunately, for the random variables $P_u$ and $T_u$ that are of interest for us, 
the effect $c$ is too large, so that Talagrand's inequality becomes useless. 
Indeed, in an extreme case 
changing the color of a single vertex $v\in N(u)$ to new color~$\lambda$
 might
result in all other vertices in $N(u)$ of color~$\lambda$ to lose their color.  This 
happens when all these vertices  have an edge pointing away from $v$.
Then, all pairs of color~$\lambda$ counted in $P_u$ are lost, and these
might be up to $r^2$. (Recall that  our graph $G$ is $r$-regular.)
Consequently, the effect $c$ has to be at least $r^2$ -- however, already an 
effect of $c\approx r$ would necessitate a deviation of 
$t\approx r\approx\bE[P_u]$ for the probability
to become small. But for Talagrand's inequality to be useful for us, we need a vanishing probability
for deviations $t$ that are small compared to $\bE[P_u]$.

Changing a single color might have a very large effect but this is a rare exception. 
Normally, only few vertices in $N(u)$ have the same color, so that also only few
are affected by any color change. That is, very large effects only occur with a very 
tiny probability. It seems unreasonable that exceedingly unlikely events should  have
a serious  impact on whether a random variable is concentrated or not.

What we need, therefore, is  a version of Talagrand's inequality
that excludes a very unlikely set $\xom$ of \emph{exceptional} outcomes
that nevertheless spoils smoothness. 
Warnke~\cite{War12} (but see also Kutin~\cite{Kut02}) extended McDiarmid's inequality 
in a similar direction by considering a sort of \emph{typical} effect. 
However, Warnke's inequality is still too weak for us. 
Grable~\cite{Gra98} as well presents a concentration inequality that excludes  exceptional 
outcomes, which would be suitable for our purposes, were it not for the fact that 
there is a serious error in its proof. 
McDiarmid~\cite{McD02}, too, describes a Talagrand-type inequality that excludes
an exceptional set. (Its main feature, though, is to allow permutations as coordinates.)
The inequality, however, does not seem to be of much use to us either.

We mention that the powerful, but technical, method of Kim and Vu~\cite{KV00} can also 
handle large but unlikely effects.

\medskip
To deal with exceptional outcomes, we modify the definition of certificates.
Given an exceptional set $\xom\subseteq\Omega$
and $s,c>0$, we say that \emph{$X$ has upward $(s,c)$-certificates} if for every $t>0$ and
for every $\omega\in\Omega\sm\xom$  
there is an index set $I$ of size at most $s$ so that 
$X(\omega')>X(\omega)-t$ for any $\omega'\in\Omega\sm\xom$ 
for which the 
restrictions $\omega|_I$ and ${\omega'}|_I$ differ in less than $t/c$ coordinates. 

Directly,
Talagrand's inequality does not give concentration around the expectation 
but around the \emph{median} $\med(X)$ of $X$, that is, around
\begin{align*}
	\med(X)=\sup\left\{t\in\mathbb R: \bP[X\leq t]\leq \frac{1}{2}\right\}.
\end{align*}
However, in typical applications the median is very close to the expected value. 
We will deal with this technicality later.
\begin{lemma}\label{lem: talagrand excep1}
Let $((\Omega_i,\Sigma_i,\bP_i))_{i=1}^n$ be probability spaces, and let
$(\Omega,\Sigma,\bP)$ be their product space.
Let $\xom\subseteq \Omega$ be a set exceptional events.
Let $X:\Omega\to\mathbb R$ be a random variable and
let $t\geq 0$.

If $X$ has upward $(s,c)$-certificates  then 
\begin{align}\label{bothbounds}
	\bP[|X- \med(X)|\geq t] \leq 4e^{-\frac{t^2}{4c^2s}}+4\bP[\xom].
\end{align}
\end{lemma}

We prove Lemma~\ref{lem: talagrand excep1} with the original, full version 
of Talagrand's inequality.
Recall that the \emph{Hamming distance} of two points $\omega,\omega'\in \Omega$
is defined by the number of non-common coordinates or equally $\sum_{i:\omega_i\not=\omega_i'}1$.
For a weighted version (together with a normalization), let $\alpha\in \mathbb R^n$ be a unit vector with $\alpha_i\geq 0$.
We define the \emph{$\alpha$-Hamming distance} between $\omega$ and $\omega'$ by $\sum_{i:\omega_i\not=\omega_i'}\alpha_i$.

For a set $A\in \Omega$ and a point $\omega\in \Omega$,
let 
$$d(\omega,A)=\sup_{\alpha}\left\{\tau:\sum_{i:\omega_i\not=\omega_i'}\alpha_i\geq \tau \text{ for all } \omega'\in A\right\},$$
which is equivalent to the largest value $\tau$
such that all points in $A$ have $\alpha$-Hamming distance at least $\tau$ to $\omega$ 
(for a best possible choice of $\alpha$).

\begin{theorem}[Talagrand \cite{Tal95}]\label{thm: talagrand}
Let $((\Omega_i,\Sigma_i,\bP_i))_{i=1}^n$ be probability spaces, and 
let $(\Omega,\Sigma,\bP)$ be their product space.
If $A,B \subseteq \Omega$ are two (measurable) sets
such that $d(\omega,A)\geq \tau$ for all $\omega\in B$,
then
\begin{align*}
	\bP[A]\bP[B] \leq e^{-\tau^2/4}.
\end{align*}
\end{theorem}

\begin{proof}[Proof of Lemma~\ref{lem: talagrand excep1}]
The two-sided estimation~\eqref{bothbounds} follows from
the two one-sided estimations
\begin{align*} 
\bP[X\leq\med(X)- t] &\leq 2e^{-\frac{t^2}{4c^2s}}+2\bP[\xom]\\
\bP[X\geq\med(X)+ t] &\leq 2e^{-\frac{t^2}{4c^2s}}+2\bP[\xom],
\end{align*}
of which we only show the first; 
the argumentation for the second is analogous.
Let 
\begin{align*}
A&=\{\omega\in \Omega\sm \xom: X(\omega)\leq \med(X)-t\}\text{ and}\\
B&=	\{\omega\in \Omega\sm \xom: X(\omega)\geq \med(X)\}.
\end{align*}
Pick an arbitrary $\omega\in B$. 
By assumption, $X$ has upward $(s,c)$-certificates, which means, in particular, 
that there is an index set $I$ of at most $s$ indices such that $\omega|_I$
and $\omega'|_I$ differ in at least $t/c$ coordinates for every $\omega'\in A$.
Consequently, if we set  $\alpha=1/{\sqrt{|I|}}\cdot\mathbbm{1}_{I}$, 
where $\mathbbm{1}_{I}$ is the characteristic vector of $I$, then 
 $\omega$ and $\omega'$ have $\alpha$-Hamming distance   at least $\frac{t}{c\sqrt{s}}$.
Hence $d(\omega,A)\geq \frac{t}{c\sqrt{s}}$.

Using Theorem \ref{thm: talagrand},
we obtain $\bP[A]\bP[B]\leq e^{-t^2/4c^2s}$.
We conclude:
\begin{align*}
	\bP[X\leq \med(X)-t]\cdot \frac{1}{2}
	&\leq\bP[X\leq \med(X)-t]\bP[X\geq \med(X)]\\
	&\leq \bP[A]\bP[B] + \bP[\xom]\\
        &\leq e^{-t^2/4c^2s}+ \bP[\xom].
\end{align*}
This completes the proof.
\end{proof}

Next, we prove that usually the median is close to the expected value:
\begin{lemma}\label{lem: med exp}
Let $((\Omega_i,\Sigma_i,\bP_i))_{i=1}^n$ be probability spaces, and 
let $(\Omega,\Sigma,\bP)$ be their product space.
Let $\xom\subseteq \Omega$.
Let $X:\Omega\to\mathbb R$ be a random variable, let $M=\max\{\sup |X|,1\}$, and let $c\geq1 $.
If  $X$ has upward $(s,c)$-certificates,  
then
\begin{align*}
	|\bE[X]-\med(X)|\leq 20 c\sqrt{s}+20M^2 \bP[\xom].
\end{align*}
\end{lemma}

We note that the proof technique is not new. A similar proof appears, for instance, 
in Molloy and Reed~\cite[Ch.~20]{MR02}.

\begin{proof}
Clearly,
\begin{align*}
	|\bE[X]-\med(X)|
	&\leq \bE[|X-\med(X)|].
\end{align*}
Put $K=\lfloor 2M/c\sqrt s\rfloor$, and observe that  
$|X-\med(X)|< (K+1)c\sqrt s$.
By splitting the possible values of $|X-\med(X)|$ into intervals of length $c\sqrt s$,
we can upper-bound 
\begin{align*}
	\bE[|X-\med(X)|]
	\leq \sum_{k=0}^K c\sqrt{s} (k+1) \bP[|X-\med(X)|\geq kc\sqrt{s}\}].
\end{align*}
We apply  Lemma~\ref{lem: talagrand excep1} for each summand with $t=kc\sqrt{s}$:
\begin{align*}
	\bE[|X-\med(X)|]&\leq c\sqrt{s}\sum_{k=0}^K 4(k+1)(e^{-k^2/4}+\bP[\xom])\\
	&\leq 20 c\sqrt{s}+20M^2 \bP[\xom],  
\end{align*}
as $\sum_{k=0}^\infty (k+1)e^{-k^2/4}\approx 4.1869<5$.
\end{proof}

We conclude under the assumptions as in Lemma~\ref{lem: med exp} by combining Lemma~\ref{lem: talagrand excep1} and~\ref{lem: med exp}:
\begin{align}\label{intermediate}
	\bP[|X- \bE[X]|\geq t+20 c\sqrt{s}+20M^2 \bP[\xom]] \leq 4e^{-\frac{t^2}{4c^2s}}+4\bP[\xom].
\end{align}

Let us come back to the certificates. As defined, they are witnesses for large values of $X$. 
Sometimes,
it is easier to certify smaller values. To capture such situations we say that 
\emph{$X$ has downward $(s,c)$-certificates}
if for every  $t>0$, and 
for every $\omega\in\Omega\sm\xom$ 
there is an index set $I$ of size at most $s$ so that $X(\omega')< X(\omega)+t$
for 
every $\omega'\in\Omega\sm\xom$ for which the 
restrictions $\omega|_I$ and ${\omega'}|_I$ differ in less than $t/c$ coordinates. 
By replacing $X$ with $-X$ we  observe that 
Lemmas~\ref{lem: talagrand excep1} and~\ref{lem: med exp}, and thus~\eqref{intermediate}, remain valid 
for downward certificates. 

We simplify~\eqref{intermediate} a bit more.
If $t\geq 50c\sqrt{s}$ and $\bP[\xom]\leq M^{-2}$ then $t\geq t/2+ (20 c\sqrt{s}+20M^2 \bP[\xom])$.
Thus:

\begin{theorem}\label{thm: talagrand nice}
Let $((\Omega_i,\Sigma_i,\bP_i))_{i=1}^n$ be probability spaces, $(\Omega,\Sigma,\bP)$ be their product, and
let $\xom\subset \Omega$ be a set of exceptional outcomes.
Let $X:\Omega\to\mathbb R$ be a random variable, let $M=\max\{\sup |X|,1\}$, and let $c\geq 1$.
Suppose $\bP[\xom]\leq M^{-2}$
and $X$ has upward $(s,c)$-certificates or downward $(s,c)$-certificates, then for $t> 50 c\sqrt{s}$
\begin{align*}
	\bP[|X- \bE[X]|\geq t] \leq 4e^{-\frac{t^2}{16c^2s}}+4\bP[\xom].
\end{align*}
\end{theorem}



\section{Concentration for coloring procedure}\label{sec: proof concentration}

We finally complete the proof of Lemma~\ref{coloring lemma} by proving that 
for every vertex $u$ the random variables $P_u$ and $T_u$ are tightly concentrated around 
their expected values. Recall that $P_u$ counts the number of vertex pairs  in $N(u)$
that at the end of the coloring procedure have the same color, and recall that $T_u$
counts the number of such vertex triples. 

\begin{proof}[Proof of Lemma~\ref{concentration}]
We show the statement for $P_u$, the proof for $T_u$ is almost the same.

Recall that the coloring procedure that defines $P_u$ (and $T_u$) is based on a 
random experiment on a product space $\Omega=\prod \Omega_i$, where the product 
ranges over the vertices (that receive a color) and the edges (that receive a direction)
of the graph. Thus, every event $\omega\in\Omega$ is a vector that is indexed by $V(G)\cup E(G)$. 

To apply Theorem~\ref{thm: talagrand nice}, we define the set $\xom$ of 
exceptional events as the space of events that assign some color to more than $\log r$
vertices in $N(u)$ (before uncoloring). 
Recall that we work with $\C=\lceil (1-\gamma) r\rceil$ colors.
To estimate $\bP[\xom]$ observe first that the probability that 
a particular color appears more than $\log r$ times in $N(u)$ is at most
\[
\sum_{i=\lfloor \log r\rfloor+1}^r\binom{r}{i}\frac{1}{\C^i}\leq 
\sum_{i=\lfloor \log r\rfloor+1}\left(\frac{er}{i}\right)^i\frac{1}{(1-\gamma)^ir^i}
\leq r\cdot \left(\frac{e}{(1-\gamma)\log r}\right)^{\log r}
\]
Thus, we get
\[
\bP[\xom] \leq r^2\cdot \left(\frac{e}{(1-\gamma)\log r}\right)^{\log r}\leq r^{-\frac{2}{3}\log\log r},
\]
for large enough $r$. 

Setting $s=3r$ and $c=\log^2r$, let us check that $P_u$ has downward $(s,c)$-certificates. 
So, let $t>0$ be given and consider an event $\omega\notin\xom$. 
We have to define an index set $I$ of size at most~$s$. 
We start by including all vertices of $N(u)$ in $I$. 
Next, consider any vertex $v\in N(u)$ that 
becomes uncolored in step~$3$ of the coloring procedure. This is only the case, if there 
is a neighbor $v'$ of $v$ that receives the same color under $\omega$ and if, again 
subject to~$\omega$, the edge $vv'$ points towards $v$. We add  $v'$ and $vv'$ 
to $I$ for each such vertex $v$. In total, we have $|I|\leq 3r=s$ as for every vertex $v\in N(u)$ 
we add at most two additional indices to $I$. 

Let now $\omega'\notin\xom$ be an event with $P_u(\omega')\geq  P_u(\omega)+t$. 
As $\omega'\notin\xom$, every color may contribute at most $\binom{\log r}{2}\leq \frac{1}{2} \log^2r$ pairs
to $P_u$. 

For every color $\lambda$
for which there are more pairs of vertices colored with $\lambda$ contributing to $P_u(\omega')$ than to $P_u(\omega)$,
there is a vertex $v$ in $N(u)$ of color $\lambda$ under $\omega'$ but
that, under $\omega$, is either uncolored or colored with a different color than $\lambda$.

In the latter case, $\omega$ and $\omega'$ differ in the coordinate $v\in I$ (as $v$ receives different colors under $\omega$ and $\omega'$).
For the former case, observe that we had added a vertex $v'$ and the edge $vv'$ to $I$ to witness $v$ being uncolored under $\omega$. Either the
color of $v'$ or the direction of $vv'$ must be different in $\omega'$, that is, $\omega$ and $\omega'$ differ again in at least
one coordinate of $I$. In both cases, call such a coordinate a \emph{$\lambda$-difference}. 

Can a coordinate in $I$ be a $\lambda$-difference and a $\mu$-difference for two different colors $\lambda$ and $\mu$?
Yes, but only if it is a vertex $v'\in I$ in $N(u)$ that satisfies two conditions: first,
under $\omega$ it serves as a witness for one of its neighbors $v\in N(u)$
losing its color $\lambda$, say, in the conflict resolution step of the coloring procedure; and second, 
by flipping from color $\lambda$ under $\omega$ to color $\mu$ in $\omega'$ the vertex $v'$ contributes new pairs of color~$\mu$
in $P_u$. This then immediately shows that no coordinate can be a $\lambda$-difference for three (or more) colors. 

As a consequence,  $\omega'$ and $\omega$ must differ in at least 
$t/\log^2r=t/c$ coordinates in $I$ as $P_u(\omega')\geq P_u(\omega)+t$.
(Recall that under $\omega'$ no color can contribute more than $\tfrac{1}{2}\log^2r$ pairs to $P_u$.)
This proves that $P_u$  has downward $(s,c)$-certificates.

For Theorem~\ref{thm: talagrand nice}, we set $M=\sup P_u\leq r^2$.
With $t=\log^3r \sqrt r$,
 Theorem~\ref{thm: talagrand nice} implies for large $r$ that 
\begin{align*}
\bP\left[|P_u-\bE[P_u]|\geq\sqrt{r}\log^3 r\,\right] 
&\leq 4\exp\left(-\frac{r\log^6r}{16\log^4 r\cdot 3r}\right) + 4r^{-\frac{2}{3}\log\log r}\\
&\leq r^{-\frac{1}{2}\log\log r}
\end{align*}
The only difference of the proof for $T_u$ lies in the fact that, outside $\xom$, 
every color can contribute up to $\log^3r/6$ triples to $T_u$; for $P_u$
this was at most $\log^2r/2$ pairs. The resulting higher value $\log^3r$ for $c$ 
can easily be compensated for by increasing $t$ to $\log^4r\sqrt r$ in the application of
Theorem~\ref{thm: talagrand nice}.
\end{proof}

\section{Exceptional outcomes spoil triangle counting}\label{sec: triangles}

We close this article by arguing that Theorem~\ref{thm: talagrand nice}
has potential applications beyond our coloring lemma. To make this case, 
we discuss the number of triangles in a random graph. We note that 
this problem also serves as a motivating example for Kim and Vu~\cite{KV00}, 
and for Warnke~\cite{War12}. 

Consider the random graphs $\mathcal G(n,p)$ that are obtained from $K_n$ by deleting uniformly and independently at random an edge with probability $1-p$, where $p=p(n)$
may be a function in $n$. The threshold probability for the triangles is $p=\tfrac{1}{n}$:
below that threshold there is with high probability no triangle, above it there is with high probability at least one. 
Let us now examine the number $T$ of triangles in $\mathcal G(n,p)$,
or rather the expected number of 
triangles $\bE[T]=\binom{n}{3}p^3\approx \tfrac{1}{6}n^3p^3$. 

Is $T$ concentrated 
around its expected value, whenever $np\to\infty$? We consider here a 
relatively mild notion of concentration, where we allow deviations from
the expected value of up to $t=\epsilon n^3p^3$ for small but fixed $\epsilon>0$.
For simpler notation, set $N=\binom{n}{2}$.

For $p$ relatively large, that is, for $p\geq n^{-1/3+ \gamma}$ for any $\gamma>0$, 
McDiarmid's inequality~\cite{McD89} is strong enough to show concentration. 
Indeed, changing any coordinate, that is, any edge, may result in at most
$n$ new triangles (or at most $n$ triangles less), so that the effect $c$
of each coordinate is bounded by $n$. Consequently, for McDiarmid's 
inequality to show that $|T-\bE[T]|\leq t$ for $t=\epsilon n^3p^3$
it is necessary that $\frac{t^2}{N\cdot n^2}$ tends to infinity.
This is the case if $p\geq n^{-1/3+ \gamma}$.
A version of McDiarmid's inequality for binary random variables, 
see~\cite{McD89} again, allows to drop this threshold to $p\geq n^{-2/5+ \gamma}$.

To go below this threshold, Warnke~\cite{War12} (but also others) observed that, while changing 
a single edge may create up to $n-2$ new triangles
(or destroy that many), this is exceedingly unlikely. 
Indeed, we only expect a particular edge to be in roughly $np^2$ many triangles.
Thus, by the standard Chernoff bound, it is, for any $\delta>0$,  extremely unlikely that an edge is contained in more than $\max\{2np^2,n^\delta\}$ triangles.


Exploiting the fact that, typically, the effect of changing a single edge is much smaller,
Warnke could verify concentration as long as $\frac{t^2}{pN\cdot \max\{2np^2,n^\delta\}}$ 
tends to infinity, which is  the case when $p\geq n^{-4/5+ \gamma}$.

To go even below such $p$, Kim and Vu~\cite{KV00} developed a powerful method
that evidently has a very wide scope of application. 
Usually, 
great power does not come for free, and this is also the case here: 
Kim and Vu's inequality is rather technical and not easy to use. 

Let us now apply Theorem~\ref{thm: talagrand nice} and show the strong concentration of $T$ also for  values of $p$ smaller than $n^{-4/5+ \gamma}$.
We exclude all outcomes where at least one edge is contained in more than $n^\delta$ many triangles. As seen above, this is an event of very small probability.
Moreover, we may use Warnke's result (or a previous
application of Theorem~\ref{thm: talagrand nice})
to observe that it is extremely unlikely that $T\geq \tfrac{2}{6}n^3\cdot n^{3(-4/5+ \gamma)}= \frac{1}{3} n^{3/5+ 3\gamma}$. (We have used here that $T$ is monotone in $p$.)
We add all these outcomes to our exceptional set.

Next, let us check that $T$ has upward $(s,c)$-certificates, where 
$s=n^{3/5+ 3\gamma}$ and $c=n^\delta$. For a non-exceptional event $\omega$, 
we use as index set $I$ the set of all edges lying in any triangle. As 
there are at most $\frac{1}{3} n^{3/5+ 3\gamma}$ triangles, it follows that $|I|\leq s$. 
Now, consider some non-exceptional event $\omega'$ such that $\omega$ and $\omega'$
differ in less than $t'/c$ coordinates of $I$. Then, any edge in $I$ that is 
present in $\omega$ but lost in $\omega'$ may only result in $\omega'$ having 
at most $n^\delta$ less triangles than $\omega$. Moreover, edges outside $I$ 
obviously cannot result in the loss of triangles. Therefore, $T(\omega')>T(\omega)-t'$,
and we see that $T$ has upward $(s,c)$-certificates.

With these values of $s$ and $c$, we deduce from Theorem~\ref{thm: talagrand nice}
that $T$ is strongly concentrated if
$p\geq n^{-9/10+ \gamma}$.

For even smaller values of $p$, we may apply Theorem~\ref{thm: talagrand nice} once again, and set $s \approx n^{-9/10+ \gamma}$, which then will yield
concentration for $p\geq n^{-19/20+ \gamma}$.
Of course, this can iterated several times, 
so that we get strong concentration for $p\geq n^{-1+ \beta}$ for every $\beta>0$.

Let us finally point out that we even have fairly tight concentration around 
the expected value: 
namely, $T$ is very likely within the range $\bE[T]\pm n^{\delta}\sqrt{\bE[T]}$.

\section{Acknowledgment}

We both would like to thank Ross Kang for insightful discussions.

\bibliographystyle{amsplain}
\bibliography{litstrchrom}

\vspace{1cm}
\mbox{}

\vfill

\small
\vskip2mm plus 1fill
\noindent
Version \today{}
\bigbreak

\noindent
Henning Bruhn
{\tt <henning.bruhn@uni-ulm.de>}\\
Felix Joos
{\tt <felix.joos@uni-ulm.de>}\\
Institut f\"ur Optimierung und Operations Research\\
Universit\"at Ulm, Ulm\\
Germany\\

\end{document}